\documentclass[reqno]{amsart}
\usepackage{color,graphicx,enumerate,wrapfig,amssymb}
\usepackage{amssymb} 
\usepackage{amsmath}
\usepackage{esint}
 \usepackage{cite}
\usepackage{comment}
\usepackage{color,graphics}
\usepackage{fixme}
\usepackage{graphicx}
\newtheorem{theorem}{Theorem}[section]
\newtheorem{lemma}[theorem]{Lemma}

\newtheorem{corollary}[theorem]{Corollary}
\theoremstyle{definition}

\theoremstyle{remark}
\newtheorem{remark}[theorem]{Remark}
\numberwithin{equation}{section}
\theoremstyle{claim}

\newcommand{\R}{{\mathbb R}}

\newcommand{\C}{{\mathcal{C}}}

\author{Mark Allen}
\address{Department of Mathematics, Brigham Young University, Provo,  UT 84602, US}
\author{Henrik Shahgholian}
\address{Department of Mathematics, KTH  Royal Institute of Technology, Stockholm, Sweden}
\thanks{Much of this work was completed while the first author visited KTH Royal Institute of Technology. Shahgholian was supported by Swedish Research Council.}

\title[Boundary Harnack]{A New Boundary Harnack Principle \\  (equations with right hand side)}

\begin{document}

\begin{abstract}
We introduce a new boundary Harnack principle in Lipschitz domains for equations with a right hand side. 
Our approach, which uses comparisons and blow-ups,  will  adapt to more general domains as well as other types of operators. We prove the principle for 
divergence form elliptic equations with lower order terms including zero order terms. The inclusion of a zero order term appears to be new even in the absence of a right hand side.   
\end{abstract}

\maketitle
\makeatletter
\vspace{-2em}
{\centering\enddoc@text}
\let\enddoc@text\empty % to remove the contact info from the end of the document
\makeatother

\section{Introduction}
 
 \subsection{Background}
The well-known  boundary  Harnack principle states that two non-negative harmonic functions are comparable close to part of the boundary of a given domain, where they both vanish. More exactly if $u$ and $v$ are harmonic functions  in $D \cap B_1 $ and vanish on $\partial D \cap B_1$, with $D$ a Lipschitz  domain,\footnote{The boundary Harnack Principle holds in very general domains, such as NTA domains, and uniform domains. It also holds for solutions to a large class of elliptic equations.} then 
$$
\frac1C v(x) \leq u(x) \leq C v(x)  , 
$$      
where $C$ depends on space dimension and $u(x^0)/v(x^0)$ for a fixed $x^0$ in the domain.

We are interested in  extending this result to the case of equations with right hand side.
Of course such a general result is doomed to fail, unless some further restrictions are imposed.  This can be  seen  through  a  simple $2$-dimensional example with 
$$u(x)=x_1x_2(x_1^2 - x_2^2) , \qquad v(x) = x_2(x_1 - x_2)  $$
 in the cone $\{ x_1 >x_2  > 0 \}$ with aperture   $\pi/4 $.  Consequently, there cannot exist\footnote{ One can actually prove  the failure of boundary Harnack between harmonic and super-harmonic functions,  in the first quadrant. This is illustrated in  Example 1.2 in \cite{s04}, in terms of free boundary problems.}
$C >0$ such that  
 $Cu \geq v $.   
Another simple (and  discouraging)  example is
$$u(x) = x_1^2, \qquad v(x) = x_1$$ 
in the set $\{x_1 > 0\}$. Again, there cannot exist a constant $C$ such that close to the boundary  $Cu \geq v$. 

There are two observations to make from the above examples:

\begin{itemize}
\item [1)] In the first  example the  domain is too narrow, with a sharp corner at the boundary.

\item [2)] In the second example we try to show subharmonic functions can dominate harmonic functions, by multiplying with a constant; this in general  fails.

\end{itemize}

To put things in perspective, let   $D= \{ x_1>0 \} \cap \{x_1 > x_2\}$Êand consider the following function 
$$w(x) = x_1 (x_1 -  x_2)     \quad \hbox{in }  D_r=D   \cap B_r(0).$$

 Now let $u $ be  the positive homogeneous harmonic function vanishing on the boundary of the cone $D$.  Then for $v= u -w$ we have  that  $ u  (x)  >   v (x) $, since  
  $w>0 $ in $D$ (by definition),  and $v >0$ since $w \approx r^2$ and $u \approx r^a $ ($a<2$),
  for $r$ small.   We also have 
$$
\Delta u =0, \qquad \Delta v = -1 \qquad \hbox{in }  D,
$$
with zero boundary values. In particular we have a boundary Harnack principle   $ u  (x)  >   v (x) $
where a harmonic function dominates a superharmonic one.

The difference between this example and the first example above is that the cone is wider. The question that naturally arises is whether such a behaviour can be structured through 
a general statement, and if so what are the conditions for such   a boundary Harnack principle.

A further observation is that when  the domain  $D$ is uniformly $C^{1,Dini}$ then a boundary Harnack principle holds, between a positive harmonic   and a super-harmonic   functions (with bounded r.h.s.), vanishing on the boundary  $ (\partial D) \cap B_1$. This is an easy consequence of Hopf's boundary point lemma  and $C^1$-regularity of solutions. Indeed,
 let  $u$ be a  non-negative  harmonic function in  $D$  and   $v$ satisfy $\Delta v = -1$ in $D$, both with zero Dirichlet data on $ (\partial D) \cap B_1$. 
 Then   by Hopf's boundary principle (applied to $u$) and 
that   $v$ is uniformly  $C^1$, we have
$$
\partial_\nu  u \geq C_0 > 0  \quad  \hbox{on } \partial D \cap B_1, \qquad \hbox{universal } C_0,
$$
$$
\partial_\nu  v \leq  C_1  \quad  \qquad  \hbox{on } \partial D, \qquad \hbox{universal } C_1,
$$
where  $\nu $ is  the interior normal direction, on  $\partial D$.
Hence for  $k$ large, the function $w_k := k u - v$ satisfies
   $$\partial_\nu  w_k= k \partial_\nu  u -  \partial_\nu  v  \geq  k C_0 - C_1 >  0 \qquad   \hbox{on } \partial D \cap B_1.$$ Hence  $ w_k >  0$   in $D$ in a small neighbourhood of the  $\partial D \cap B_1$ and inside $D$; this is exactly the  boundary Harnack we asked for.

\subsection{Main Result}

Our main result is the boundary Harnack principle mentioned above between a superharmonic and a positive  harmonic function, in Lipschitz  domains. 
For instructional reasons and for the benefit of the non-expert reader,
  we first   state  and prove  the theorem inside a  cone; see Theorem \ref{t:cone}.  The proof in this case presents the main ideas for the general case. Next, we do the same for a Lipschitz domain; see Theorem \ref{t:mainrhs}. We  then apply the ideas to divergence operators, see Theorem \ref{t:genrhs}. Our result holds even when including zero-order terms. 
  This result appears to be new even when considering the standard boundary Harnack principle without right hand side, i.e. for solutions rather than supersolutions. 
  
  In proving our new boundary Harnack principle, we will often utilize the boundary Harnack principle (without right hand side) for two nonnegative solutions. 
  To avoid confusion we will always refer to this as the ``standard boundary Harnack principle''.

  \subsection{Related results}
As previously mentioned the boundary Harnack principle has been proven on very general domains as well as for various operators. Recently, De Silva and Savin in \cite{ds15}
generalized the result in a new direction by showing that improving the regularity of the boundary improves the regularity of the boundary Harnack inequality.  Their result 
applies to elliptic operators in divergence form with appropriate smoothness assumptions on the coefficients. Roughly speaking, if $u,v$ vanish on $\partial \Omega$ with 
$\mathcal{L}u=0$ and $\mathcal{L}v=f$
in $\Omega$, and if $u>0$ in $\Omega$ and 
$\partial \Omega \in C^{k,\beta}$, then
\begin{equation}   \label{e:ds}
\left\| \frac{v}{u} \right\|_{C^{k,\beta}}  \leq C \left(\| v \|_{L^{\infty}} + \| f  \|_{C^{k-1,\beta}} \right).
\end{equation}
As a corollary of Theorem \ref{t:mainrhs} we obtain a similar estimate on Lipschitz domains with small enough Lipschitz norm, see Corollary \ref{c:mainrhs}.  

The result in \cite{ds15} illustrates the principle that improving the regularity of the boundary gives a boundary Harnack principle for solutions with a right hand side. Our result 
shows that Lipschitz regularity is a sufficient condition to obtain an estimate of the form \eqref{e:ds} when $\beta=0$, and the allowed behavior for $f$ is determined by the Lipschitz constant. 
 
 Another related result is found in \cite{s96} where it is shown a superharmonic function is comparable to the first eigenfunction for a domain in $\mathbb{R}^2$ with finitely many corners and with
 an interior cone condition.

\subsection{Applications}
We present two applications of our boundary Harnack principle: to the Hele-Shaw flow and to the obstacle problem. 
\subsubsection{Hele-Shaw Flow} 
The Hele-Shaw flow may be formulated as follows:
For $t >0$ we define $u^t(x) $ as the solution of
$$
\Delta u^t (x) = \chi_{\Omega^t} - \chi_{\Omega^0 } - t\delta_z \qquad \hbox{in } \R^n,
$$
where $\Omega^0$ is the initial domain (filled with liquid) and $\Omega^t = \{u^t > 0\} $, and $z \in \Omega^0$ is the liquid-injection point. The Dirac  mass at the point $z$ means we inject more fluid at that point.
 This formulation is obtained after a Baiocchi-transformation, which is easily found in the classical  literature for  free boundary problems.

Now suppose we do the Hele-Shaw experiment on a ``table" having   corners of various 
angels. More exactly suppose we consider 
$$
\Delta u^t (x) = \chi_{\Omega^t} - \chi_{\Omega^0 } - t\delta_z \quad \hbox{in } D, \qquad u^t (x) = 0 
\quad \hbox{on } \partial  D,
$$
where $D \supset \Omega^0$ is a given domain. 
The zero boundary data on $\partial D$ means that the liquid, when reaching the edge of the table, will fall to the ground.
The question is whether the liquid will reach all points of the boundary of $D$ in finite time $t\leq t_0< \infty$.

By a  barrier argument (see \cite{s04}) one can show that in two dimensions corners with angle  smaller than 
 or equal to  $\pi/2$ will not   get wet; an analogous result for higher dimensions is a consequence of   Theorem \ref{t:counter} in this paper. 
Now the question is what can happen when the angle of the corner is larger than $\pi/2$; will the liquid reach such a corner?  The answer to this question is yes, see \cite{s04}. 

We now consider $D$ with Lipschitz boundary with Lipschitz constant $L< 1$. As a consequence of Theorem \ref{t:mainrhs} we now show that every point of the boundary will get wet.  We need to show for large values of $t$ that $u^t > 0$ in $D \cap B_r (z^1)$, for any $z^1 \in  \partial D$. We write $u^t (x) = h^t (x) - k(x)$ 
where $h^t $ is the harmonic function in $B_r(z^1) \cap D$ with boundary values $h^t= u^t$, and  $k$ ($t$-independent) satisfies
$$
\Delta k = -1 \quad \hbox{ in } B_r(z^1) \cap D, 
$$
with zero boundary values on $\partial ( B_r(z^1) \cap D)$. 
It suffices then to show 
 that for large values of $t$ we have 
$$
h^t \geq k \qquad \hbox{ in } B_r(z^1) \cap D.
$$

By Theorem \ref{t:mainrhs}, for fixed $t$ there exists some large constant $C_t$ such that $C_t h^t \geq k$. By the standard  boundary Harnack principle 
\[
 \sup_{B_{r/2}(z^1)\cap D} \frac{h^t}{h^s} \leq C \frac{h^t(z^2)}{h^s(z^2)},
\]
for a fixed $z^2 \in B_r(z^1)\cap D$. 
Since $\lim_{s \to \infty} h^s(z^2) = \infty$, we choose $s$ large enough so that $h^t(z^2)/h^s(z^2) < C^{-1}/C_t $. 
Then 
\[
 \sup_{B_{r/2}(z^1) \cap D}\frac{h^t}{h^s} \leq C\frac{h^t(z^2)}{h^s(z^2)} \leq \frac{C C^{-1}}{C_t}, 
\]
so that $h^s \geq C_t h^t \geq k$ to conclude that $u^t>0$ in $B_{r/2}(z^1) \cap D$.

\bigskip

A related question to the Hele-Shaw flow reaching corners, is whether for 
 $D \subset \R^n$, with $0  \in \partial D$, and $ \partial B_1 \cap D \not =\emptyset $, 
a solution to 
$$
\Delta v_m = 1 \quad  \hbox{ in } D \cap B_1, \qquad v_m=0 \hbox{ on } \partial D \cap B_1
\qquad  v_m=m \hbox{ on } \partial B_1 \cap D,
$$
will be non-negative in $D \cap B_1$ for $m$ large enough. The answer to this question is already given in the  discussion for Hele-Shaw experiment above.

\subsubsection{Obstacle Problem}
A further application can be made to  the regularity theory of the free boundary  for the obstacle problem, that is formulated as a solution to   
$$
\Delta  v =  h\chi_{\{v > 0\}}, \qquad v  \geq 0 \qquad \hbox{in } B_1,
$$
with $h \geq c_0 >0 $ Lipschitz, and a  prescribed  Dirichlet data on $\partial B_1$.
 In proving regularity of the free boundary for the obstacle problem  
one can show that if a free boundary point $z \in \partial \{v > 0\} \cap B_{1/2}$ is not a cusp point, then for some $r>0$ and  direction $e$,   $v_e > 0$   in the set $\{ v> 0\} \cap B_r(z) $,   
  and that the free boundary is Lipschitz in $B_r(z) $. 

One can actually show that Lipschitz norm can be taken as small as one  wishes, by taking the neighbourhood of $z$ smaller. The proof for (non-uniform) Lipschitz regularity is actually much simpler than proving uniform regularity. 
We refer to \cite{psu12}  for  background and details as well as other related original references. 

The boundary Harnack principle in this paper allows us to deduce $C^{1,\alpha}$-regularity of the free boundary for the obstacle problem, in an elementary way.\footnote{This was done for constant $h$ in \cite{ath-caff-1985} using the standard boundary Harnak principle for Lipschitz domains.}
  Indeed, we may consider the function $H(x) =  v_{e_1} - C v  $, which satisfies\footnote{Actually this conclusion is part of proving the Lipschitz regularity of the free boundary; see \cite{caffarelli:obstacle-revisited}
}
$$
H > 0 , \qquad  v_{e_1} > 0 , \qquad    \Delta H  = h_{e_1} - C h \leq 0   \qquad \hbox{in } \{v >0\} \cap B_r(z).
$$
We will now apply Corollary \ref{c:mainrhs} to the  functions $H=v_{e_1}-Cv$ and $v_e$ for any direction $e$ orthogonal to $e_1$. Both $H$ and $v_e$ vanish on the free
boundary $\partial \{v>0\}$. 
By taking a neighborhood of $z \in \partial \{v > 0\} \cap B_{1/2}$ small enough, the Lipschitz norm of $\partial \{v > 0\} \cap B_{1/2}$ can be made arbitrarily small. 
 Since $\Delta H \leq 0$ and $|\Delta H|$ is bounded, we may apply 
Corollary  \ref{c:mainrhs} with $\gamma=0$ to conclude that 
\begin{equation}\label{quotient}
\frac{ v_e }{ v_{e_1}  - Cv } 
\end{equation}
is $C^\alpha$ inside $\{v >0\}$ (close to the free boundary point $z$).

 Next fix a level surface\footnote{The level surface is smooth since $v_{e_1} > 0$ there.}   $v=l$, and denote the free boundary as a Lipschitz graph $x_1=G(x')$.  
 Consider  $v(G(x') , x') = l$, which after differentiation in the $e$-direction   gives
$$
 -G_e  = \frac{ v_e }{ v_{e_1} } , 
$$
inserting this into \eqref{quotient} gives us 

$$
\frac{ v_e }{ v_{e_1}  - Cv } =  \frac{ \frac{ v_e}{ v_{e_1} } }{ 1   - \frac{ Cv}{  v_{e_1}  }} =  
\frac{-G_e  }{ 1   - \frac{ Cl}{  v_{e_1}  }  } ,
$$
is $C^\alpha$, independent of $l$.  Since\footnote{See e.g. \cite{psu12}}
 $v_{e_1} \approx \sqrt{l}$ we have that 
 $$
\frac{ v_e }{ v_{e_1}  - Cv } =
\frac{-G_e  }{ 1   - \frac{ Cl}{  v_{e_1}  }  }   \to -G_e,
$$
as $l \to 0$. Hence $G_e$ is $C^\alpha$.

%
%  \bblue{We should remark that the argument above applies to div-equations with lower order terms
%  but enough regularity of the coefficients. I guess we need to double check the computations to make sure.}
  
   \subsection{Future directions}
 It seems plausible that the results presented in this paper can be generalized to other  operators, 
as well as more complicated domains. Here we have chosen to treat the problem in Lipschitz domains only. In the final section we consider second order elliptic equations of 
divergence form. The coefficients are variable and assumed to be only bounded and measurable.
%\mpar{needs change here. Also check Ref. in our Ref. folder $Riahi2001_Article_BoundaryBehaviourOfPositiveSol$}

Key elements of our approach is the standard boundary Harnack principle, barrier arguments, as well as scaling and blow-up invariance. Since our approach is indirect and uses scalings, the core idea is to look at nonnegative solutions on global domains.  The technical difficulties that seem to arise for generalizitions of our result concern the invariance of the domains in scaling.

The methods presented should also work to prove a boundary Harnack principle 
for the positivity set of a solution to the thin obstacle problem as long as it is assumed
a priori that the free boundary is Lipschitz. Then  we may argue 
in a similar way as above  for the thin obstacle problem, with equations having Lipschitz right hand side; see \cite{ath-caff-1985} in combination with our result.

%\bblue{There are  more applications: Recent paper by Caff-Sh.Yerissian, System case, and also the paper by Caff-Aguirela-Spruck ,  Ed. Teixeira did something also, Optim. problem for heat conduction. 
%
%Another application is singular perturbation
%
%$$
%\Delta u^\epsilon = \frac{f(x)}{\epsilon} \chi_{ \{ 0 < u^\epsilon< \epsilon \} }
%$$
%Assume the boundary is  Lipschitz (small Lip norm) close to a free boundary point $z$, and in $B_r(z)$.
%In this case we may scale the problem  at a free boundary point $z$, 
%$$\tilde u^\epsilon (x) = \frac{u^\epsilon (\epsilon x + z)}{\epsilon}$$
%so that
%$$
%\Delta \tilde u^\epsilon =  \chi_{\{ 0 < \tilde u^\epsilon< 1 \}}
%$$
%Now applying our BHP we get that the scaled FB is smooth in $B_{r_0} (0)$ and hence the original one is smooth (with some $1/\epsilon$ parameter). 
%
%I guess this would mean that the smoothness is independent of $\epsilon$, since otherwise a blow-up of such points, will give us a global FB problem, 
%$$
%\Delta \tilde u^0 =  \chi_{\{ 0 < \tilde u^0< 1 \}}
%$$
%and with the Bdry being a smooth graph. but not a plane. This would be a contradiction to classification of global sol. (which has to be done)
%Here I am thinking that I am scaling so that we do not end up with plane, and this can be done if the smoothness gets worse for smaller $\epsilon$. 
%
%OK, This might be some more detail needed, and maybe we can just mention it... Or leave it.
%}
%

%%%%%%%%%%%%%%%%
%%%%%%%%%%%%%%%%
%%%       CONES 
%%%%%%%%%%%%%%%%
%%%%%%%%%%%%%%%%

\section{Boundary Harnack in  Cones}  \label{s:cone}
 We let $\mathcal{C}$ be any open cone in $\R^n$, with vertex at the origin such that 
 $\mathcal{C}\cap \mathbb S^{n-1}$ is connected. If $u$ is any harmonic function on $\C$ with $u=0$ on $\partial \C$, then in 
 spherical coordinates 
  \begin{equation}\label{expansion}
     u(r,\theta) = \sum_{k=1}^{\infty} r^{\alpha_k} f_k(\theta),
  \end{equation}
  where $f_k$ are the eigenfunctions to the Laplace-Beltrami on $\C \cap \partial B_1$. 
  If $u$ is harmonic on $\C$ and nonnegative, then $u$ is unique up to multiplicative constant to $r^{\alpha_1}f_1(\theta)$.

%%%%%%%%%%%%%%%%
%  Theorem CONES
%%%%%%%%%%%%%%%%

 \begin{theorem}   \label{t:cone}
  Let $\C$ be a  connected open cone in $\R^n$ with $\C \cap \mathbb S^{n-1}$ an 
  $(n-2)$-dimensional 
  $C^{1,\alpha}$ submanifold.\footnote{
  The assumption that $\partial \C \cap \mathbb S^{n-1}$ is an $(n-2)$-dimensional manifold of class $C^{1,\alpha}$ is not necessary. 
  As we will see in Section \ref{s:lipschitz}, $\partial \C \cap \mathbb S^{n-1}$ may be a Lipshitz manifold provided the Lipschitz constant
  is small enough depending on $\gamma$, that appears in \eqref{eq:v}} Let $r^{\alpha_1}f_1(\theta)\geq0$ be harmonic in $\C$ with 
  zero Dirichlet boundary data.  Let $u$ solve
   \[
    \begin{aligned}
      \Delta u&=0  \text{ in } \C \cap B_1 , \\
      u &\geq 0   \text{ in }  \C \cap B_1 , \\
      u&=0  \text{ on }  \partial \C \cap B_1, 
    \end{aligned}
   \]
   and let $v$ satisfy  
        \begin{equation}\label{eq:v}
    \begin{aligned}
      0 \geq \Delta &v(x)\geq -C_0|x|^{\gamma}  & \ &\text{ in } \C \cap B_1 , \\
      v&=0  & \ &\text{ on }  \partial \C \cap B_1 , \\
      |v|& \leq C_0 & \ &\text{ in } \C\cap B_1,
    \end{aligned}
       \end{equation}
   with $2-\alpha_1+\gamma>0$. If $x^0 \in \C \cap B_1$, then there exists a constant $C$ depending only on 
   $\C,2-\alpha_1 + \gamma$, dimension $n$, and dist$(x^0, \partial(\C \cap B_1))$ such that 
   \[
      \frac{v(x)}{u(x)} \leq C \frac{v(x^0)}{u(x^0)}   \ \text{ for any } x \in \C \cap B_{1/2}. 
   \]
 \end{theorem}
 
 Theorem \ref{t:cone} is a boundary Harnack principle, but with a right hand side. Clearly, a harmonic solution will 
 control a subsolution. The significance of Theorem \ref{t:cone} is that a harmonic solution can control a supersolution, and 
 that the allowed behavior for the right hand side depends on the opening of the cone or more explicitly on $\alpha_1$. 
 When the opening of the cone is large (so that $\alpha_1$ is small), then negative values for $\gamma$ are allowed, and the right
 hand side can have singular behavior near the boundary. When the opening of the cone is small (so that $\alpha_1$ is large), 
 then $\gamma$ must be positive and large, so that the right hand side must decay as it approaches the boundary.

In order to prove Theorem \ref{t:cone}, we will need the following convergence result. 

\begin{lemma}  \label{l:seq}
 Let $\C$ be an open cone with $\partial \C \cap \mathbb S^{n-1}$ an $(n-2)$-dimensional Lipschitz submanifold.  Fix $0< \epsilon < R$ and let  $\gamma >\alpha_1 - 2$. For any sequence $v_k$ satisfying   
 \[
   \begin{aligned}
     |\Delta v_k| &\leq C_0 |x|^{\gamma} &\text{ in } \C \cap B_R , \\
     |v_k| &\leq C_0  &\text{ in } \C \cap B_R , \\
     v_k &=0 &\text{ on } \partial \C \cap B_R , \\
   \end{aligned}
  \]
 then there exists a subsequence $v_k \to v$ uniformly on $\C \cap B_{R-\epsilon}$ with $v$ inheriting the above properties. 
 \end{lemma}

\begin{proof} 
 Since $\C \cap \partial B_R$ is a Lipschitz submanifold, $\C$ is compactly contained in a slightly larger cone $\C' \supset \C $, and
  the unique nonnegative harmonic solution with zero Dirichlet data on  $\partial \C'  $  is given as $r^{\alpha_1-\delta} f_{\delta}$.  
  We note that $f_{\delta}(\theta)$   is uniformly bounded away from $0$  on $\C$. 
 We let $v=M_1 r^{\alpha_1-2\delta}f_{\delta}$  with  $M_1$ chosen later. Then  
 \[
 \Delta v= M_1[(\alpha_1-2\delta)(\alpha_1-2\delta+n-2) -(\alpha_1+\delta)(\alpha_1+\delta+n-2)]r^{\alpha_1-2}f_{\delta}. 
 \]
 We note that $\alpha_1-2-2\delta< \gamma$, so that if $M_1$ is chosen large enough, 
 \[
  \Delta v \leq -C_0 |x|^{\alpha_1-2-2\delta} \leq -C_0 |x|^{\gamma}
 \]
 Using $v$ as a barrier, the convergence result will 
 follow by standard techniques. 
 \end{proof}

 An alternate proof of the above lemma,  for a more general domain and more general operator, is given in the proof of  Lemma \ref{l:3conv}.

%%%%%%%%%%%%
 %  Proof fo Thm CONES
 %%%%%%%%%%%%

 We now give a proof of the main theorem in this Section. 
 \begin{proof}[Proof of Theorem \ref{t:cone}]
  Fix $x^0 \in \mathcal{C}\cap B_{1/2}$, and  consider the nonnegative homogeneous solution 
  $u=r^{\alpha_1}f_1(\theta)$. By the standard boundary Harnack principle, any solution $u$ 
  as given in the statement of Theorem \ref{t:cone} will be comparable, and consequently bounded from  
  below up to a multiplicative constant, by $r^{\alpha_1}f(\theta)$. Thus, we consider $u=r^{\alpha_1}f(\theta)$. Furthermore, any 
  function $v$ as given in the statement of Theorem \ref{t:cone} will be bounded from above by a constant multiple of the superharmonic function 
  defined by 
   \[
    \begin{aligned}
     \Delta v &=-|x|^{\gamma} &\text{ in } \C\cap B_1 , \\
     v&\geq 0 &\text{ in } \C \cap B_1 , \\
     v &=0 &\text{ on } \partial \C \cap B_1 , \\
     v &=1 &\text{ on } \C \cap \partial B_1  .
    \end{aligned}
   \]
  Thus, it suffices to prove the theorem for $v$ as defined above and with $u=r^{\alpha_1}f(\theta)$. 
    
   In the following we use $r$ as a scaling parameter which may coincide with $r$ as the polar axis for homogeneous functions. 
   We first show that there exists some constant $C$ 
 such that $v(rx^0)\leq Cu(rx^0)$ for 
   all $0<r\leq 1/2$.
 Suppose by way of contradiction that no such $C$ exists, so that if  $v(t x^0)= t^{\alpha_1} h(t)$, then 
    $\limsup_{t \to 0} h(t)=\infty$. We note that $h$ is continuous away from the origin since $v$ is continuous. 
    For $0<r<1$, we consider the rescaled functions
   \[
   w_r(x) := \frac{v(rx)-\frac{v(rx^0)}{u(rx^0)}u(rx)}{\sup_{B_1 \cap \mathcal{C}} |v(rx)-\frac{v(rx^0)}{u(rx^0)}u(rx)|},
   \]
   defined on $B_{1/r}$. We point out that $w_r(x^0)=0$ and $\sup_{B_1 \cap \mathcal{C}} |w_r|=1$. 
   We also define 
   \[
    a_k := \sup_{B_1 \cap \mathcal{C}} \left|v(2^{-k}x)-\frac{v(2^{-k}x^0)}{u(2^{-k}x^0)}u(2^{-k}x)\right| 2^{-k\alpha_1}. 
   \]
   By letting $x=tx^0$, we have 
   \[
    v(rx)-\frac{v(rx^0)}{u(rx^0)}u(rx) = [r^{\alpha_1} h(tr)-h(r)t^{\alpha_1}]r^{\alpha_1}. 
   \]
   If $r_k = 2^{-k}$, and if $1/2\leq t \leq 1$ we have
   \[
    \text{osc}_{[r_k/2,r_k]} h \leq a_k. 
   \]
   Since $\limsup_{t \to 0} h(t) = +\infty $, it follows that 
   \begin{equation}  \label{e:diverge}
    \sum a_k = + \infty.
   \end{equation}
   By the definition of $w_r$ we have
   \[
    \sup_{B_{2^j}} |w_{r_k}(x)| = \frac{\sup_{B_{2^j}} |v(r_k x) - \frac{v( r_k x^0)}{u( r_k x^0)}u(r_k x)|}{a_k r_k^{\alpha_1}}          
   \]
   We use the triangle inequality to obtain 
   \begin{equation}   \label{e:seq1}
   \begin{aligned}
    |v(r_k x) - \frac{v(r_k x^0)}{u(r_k x^0)}u(r_k x)|
    &\leq | v(r_k x) - \frac{v(2^j r_k x^0)}{u(2^j r_k x^0)}u(r_k x)| \\
      & \ + \sum_{i=0}^{j-1} \left| \frac{v(2^{j-i} r_k x^0)}{u(2^{j-i} r_k x^0)}u(r_k x) - \frac{v(2^{j-i-1} r_k x^0)}{u(2^{j-i-1} r_k x^0)}u(r_k x) \right|. 
    \end{aligned}
   \end{equation}
   To bound the first term in the above inequality, we have by definition that 
   \begin{equation}   \label{e:seq2}
    \sup_{B_{2^j}} | v(r_k x) - \frac{v(2^j r_k x^0)}{u(2^j r_k x^0)}u(r_k x)| = a_{k-j} (2^j r_k)^{\alpha_1}.
   \end{equation}
   To bound the second term in the inequality, we note that 
   \[
   \begin{aligned}
   & \left| \frac{v(2^{j-i} r_k x^0)}{u(2^{j-i} r_k x^0)}u(2^{j-i-1}r_k x^0) - \frac{v(2^{j-i-1} r_k x^0)}{u(2^{j-i-1} r_k x^0)}u(2^{j-i-1}r_k x^0) \right| \\
   &\  = \left| \frac{v(2^{j-i} r_k x^0)}{u(2^{j-i} r_k x^0)}u(r_k 2^{j-i-1}x^0) - v(2^{j-i-1} r_k x^0) \right|  \\
   & \leq a_{k-(j-i)} (2^{j-i}r_k)^{\alpha_1}. 
    \end{aligned}
   \]
    Also, since there exists $C$ depending on $x^0$ such that if $|x|=1$ then $u(rx)\leq C u(rx^0)$. If we utilize the homogeneity of $u$, we conclude that 
    \[
     \sup_{B_{2^j}} u(r_kx) \leq Cu(2^j r_k x^0). 
    \] 
    We use this and the homogeneity of $u$ to obtain 
    \begin{equation}  \label{e:seq3}
   \begin{aligned}
    &\sup_{B_{2^j}} \left| \frac{v(2^{j-i} r_k x^0)}{u(2^{j-i} r_k x^0)}u(r_k x) - \frac{v(2^{j-i-1} r_k x^0)}{u(2^{j-i-1} r_k x^0)}u(r_k x) \right| \\
    &=  \sup_{B_{2^j}} \frac{u(r_k x)}{u(2^{j-i-1} r_k x^0)}  \\
    & \quad \times 
    \left| \frac{v(2^{j-i} r_k x^0)}{u(2^{j-i} r_k x^0)}u(2^{j-i-1}r_k x^0) - v(2^{j-i-1} r_k x^0) \right|  \\   
    & \leq \frac{Cu(2^j r_k x^0)}{u(2^{j-i-1} r_k x^0)} a_{k-(j-i)} (2^{j-i} r_k)^{\alpha_1}  \\
    & = C a_{k-(j-i)} (2^{j+1} r_k)^{\alpha_1}
   \end{aligned}
   \end{equation}

   We will now bound the Laplacian of $w_{r_k}$. From \eqref{e:diverge} we may apply Lemma \ref{l:sumdiv} to the sequence $\{a_k\}$ to conclude that there exists a subsequence $k_l$ such that for any $j \in \mathbb{N}$, 
   \[
    \limsup_{k_l \to \infty} \frac{\sum_{i=1}^j a_{k_l -i}}{a_{k_l}}  \leq j. 
   \]
   By choosing $r_{k_l}$ and  combining the above inequality with \eqref{e:seq1}, \eqref{e:seq2}, and \eqref{e:seq3} we conclude 
   \begin{equation}  \label{e:logbound}
    \sup_{B_{2^j}} |w_{r_k}(x)| \leq C j 2^{j\alpha_1}. 
   \end{equation}

   From the choice of $a_{k_l}$ in Lemma \ref{l:sumdiv} and the fact that $\sum a_k = + \infty$ it follows that eventually $a_{k_l}\geq l_k^{-2}$, so that 
   if $r_{k_l} = 2^{-k_l}$, then 
   \[
    a_{k_l} \geq [\ln(1/r_k)]^2. 
   \]
   
     We now use the assumption $2-\alpha_1 + \gamma>0$. Since $\Delta u=0$ and $\Delta v=-|x|^\gamma$ we have that for a sequence $r_{k_l} \to 0$ there holds 
  
     \begin{equation}  \label{e:har}
      |\Delta w_k| \leq \frac{r_k^{2+\gamma}|x|^{\gamma}}{\sup_{B_1 \cap \mathcal{C}} |v(rx)-\frac{v(rx^0)}{u(rx^0)}u(rx)|}
      \leq C r_k^{2-\alpha_1+\gamma}[\ln(1/r_k)]^2 |x|^{\gamma}\quad   \to \ 0 ,
     \end{equation}
     as long as $2-\alpha_1 + \gamma>0$.
   
     By Lemma \ref{l:seq} there exists a subsequence $w_k \to w$ uniformly in $B_R \cap \mathcal{C}$ for any $R>0$. 
     Furthermore $w$ will satisfy 
     \[
      \begin{aligned}
       &(i) &\Delta w =0   &\qquad \text{ from } \eqref{e:har} ,\\
       &(ii) &\sup_{B_1 \cap \C} |w| =1&\  ,\\
       &(iii) &w(x^0) =0  ,  \\
       &(iv) & w(x) \leq C|x|^{\alpha_1} \ln(|x|+1) &\qquad \text{ for } |x| \geq 1 \text{ from }   \eqref{e:logbound}.   
       \end{aligned}
     \] 
     By property $(ii)$ we have that $w$ is not identically zero. By property $(iii)$ we have that $w$ changes sign, so that by \eqref{expansion}
     we have $\sup_{B_R} |w| \geq c R^{\alpha_2}$ for $R \geq 1$. Since $|x|^{\alpha_1} \ln(|x| +1) < |x|^{\alpha_2}$, for large $x$, this contradicts property $(iv)$.

     %%%%%%%%%%%%
    %  Proof away from the ray tx^0
    %%%%%%%%%%%%     
    Thus, we have shown that for any $x^0 \in \C$ there exists a constant $C$ depending on $x^0$ such that 
     \[
      v(rx^0) \leq C u(rx^0) \text{ for any } 0<r\leq 1/2.  
     \]
    For  any other point $x^1 \in \C \cap B_{1/2}$, we  rescale by $v(x|x^1|/2)$, and note that 
      $|\Delta v(x|x^1|/2)| \leq |x^1/2|^2 |x|^{\gamma}$, so that    $|\Delta v(x|x^1|/2)| \leq C_1 $
   in $\C \cap (B_{3/4} \setminus B_{1/4})$. We now restrict ourselves to the situation in which $x^1/|x^1|$ is uniformly bounded away from $\partial \C$. 
   Since $v(x^0 |x^1|/2)\leq C_2 |x^1|^{\alpha_1}$, 
    we use the (interior) Harnack inequality on $\C \cap (B_{3/4} \setminus B_{1/4})$ to conclude that  
     \[
       v(x^1) \leq C_3(C_2 |x^1|^{\alpha_1} + C_1 |x^1|^{2+\gamma})  \leq C_4 |x^1|^{\alpha_1},
     \] 
     where in the second inequality we have used $2-\alpha_1 + \gamma>0$, and $|x^1| < 1$.
     
     The constant $C_3$ does not remain bounded as $x^1$ approaches $\partial \C$. Furthermore, we also have that $f_1(x^1/ |x^1|)$ goes to zero as 
     $x^1/|x^1|$ approaches $\partial \C \cap S^{n-1}$. If $x^1$ is close to $\partial \C$, 
     we employ the boundary Harnack principle \eqref{e:ds} with right hand side as given in \cite{ds15} which is applicable as long   
        as $\partial \C \cap (B_{3/4}\setminus B_{1/4}) \in C^{1,\beta}$. We may then conclude that 
     \[
     v(x^1) \leq C_5 u(x^1)(\| v \|_{L^{\infty}} + |x_1|^{2+\gamma}) \leq  C_6 u(x^1) = C_6 |x^1|^{\alpha_1}f_1(x^1/|x^1|),
     \]
     and this concludes the result.

 \end{proof}

 We now show that the assumption that $2-\alpha_1 + \gamma>0$ is essential. 
 We first consider the easier case when 
 $2-\alpha_1+\gamma<0$, and show that Theorem \ref{t:cone} cannot possibly hold. For clarity of exposition  we restrict the analysis  to  the case when $\gamma=0$, so that the right hand side is constant. 
 
 \begin{theorem}   \label{t:counter}
   Let $u$ and $v$ be as in the statement of Theorem \ref{t:cone}, and assume $\Delta v = - 1$ with 
   $2-\alpha_1  <0$. Then for any $C>0$, there exists $\rho>0$ such that 
   \[
    Cu < v \text{ in }  \C \cap B_{\rho}. 
   \]
 \end{theorem}
 
 \begin{proof}
  We note that by the standard boundary Harnack inequality, it is enough to consider $u=r^{\alpha_1}f_1(\theta)$. We normalize $f_1$ so that 
  $\sup f_1 =1$. Fix $C>0$, and let $z \in B_r \cap \C$.  Define $h(x):=|x-z|^2/(2n)$. 
  We note that $h$ is constant on $\partial B_r(z)$. If there exists $y \in \partial B_r(z)$ with $h(y)\leq Cu(y)-v(y)$, then
  \[
   \frac{r^2}{2n} = h(y) \leq Cu(y)-v(y) \leq Cr^{\alpha_1}. 
  \]
  For small enough $r$, the above inequality cannot hold since $2<\alpha_1$. Then for small enough $r$, we have
  $h \geq Cu-v$ on $\partial B_r(z)$, and hence also on $\partial(B_r(z)\cap \C)$. By the comparison principle 
  $h \geq Cu-v$ in $B_r(z)\cap \C$, and so $0=h(0)\geq Cu(z) - v(z)$, so that $Cu(z)-v(z)\leq 0$. 
  Since $C$ was arbitrary, for any $\delta >0$,  we may choose $r_0$ such that if $r<r_0$ and $x \in B_r \cap \C$, then 
  $(C+\delta)u(x)-v(x)\leq 0$. Since $u>0$ in $B_r \cap \C$ it follows that $Cu(x)-v(x) <0$ for any 
  $x \in B_r \cap \C$. 
 \end{proof}
We can  show that Theorem \ref{t:cone} is sharp
 by considering the critical case when $2-\alpha_1 =0$.  When dimension $n=2$ this result was shown in \cite{s04}. 
 \begin{theorem}   \label{t:sharp}
  Let $\C$ be a cone in $\R^n$ with $\alpha_1=2$. Then the boundary Harnack principle  with right hand side does not hold. 
 \end{theorem}
 
 \begin{proof}
  Let $v$ be as in the statement of the Theorem, and let $u=r^{\alpha_1}f_1(\theta)$. Suppose that there exists $C>0$ such that 
  \begin{equation}  \label{e:boundary Harnack principle w}
    C^{-1}u(x)\leq v(x)\leq Cu(x) \text{ for all } x \in \C \cap B_{1/2}. 
  \end{equation}
 We now use a Weiss-type monotonicity formula for superharmonic functions as in \cite{mw07}. The function
     \[
      W(r,v) := \frac{1}{r^{n+2}} \int_{B_r \cap \C} (|\nabla v|^2 -2v) \ - \frac{2}{r^{n+3}} \int_{\partial B_r \cap \C} v^2, 
     \]
     is monotonically increasing in $r$ and is constant if and only if $v$ is homogeneous. 
     We now consider the rescaled functions $v_r(x):=v(rx)/r^2$. From \eqref{e:boundary Harnack principle w} 
     and the fact that $u$ is homogeneous of degree $2$ we have that for any fixed $x \in \C \cap B_{1}$, 
     there exists $C_x$ such that 
     \begin{equation}  \label{e:nondegen}
      C_x^{-1} r^2 \leq v_r(x) \leq C_x r^2 \text{  for any } 0<r<1. 
     \end{equation}
     From Lemma \ref{l:seq} we obtain that for a sequence $r_k \to 0$, $v_{r_k} \to v_0$ uniformly in $\C \cap B_{R}$ for 
     any $R>0$. Furthermore, we will show that $v_0$ satisfies 
      \[
       \begin{aligned}
        (i)& \quad v_0 \text{ is homogeneous of degree } 2 , \\
        (ii)& \quad v_0 \geq 0  ,\\
        (iii)& \quad v_0 =0 \text{ on } \partial \C  ,\\
        (iv)& \quad \Delta v_0 = -1 \text{ in } \C  ,\\
        (v)& \quad v \text{ is not identically zero.}
       \end{aligned}
      \]
      Property $(i)$ follows from the Weiss-type monotonicity formula in the following manner. One may 
      easily check that $W(\rho r,v)=W(\rho, v_r)$. Since $W(r,v)$ is monotone in $r$ it follows that 
      \[
      W(\rho, v_0) = \lim_{r_k \to 0} W(\rho, v_{r_k}) = \lim_{r_k \to 0} W(\rho r_k ,v) = W(0+,v) \text{  for any } \rho>0. 
      \]
     Since $W(\rho, v_0)$ is constant, then $v_0$ is homogeneous of degree $2$, see \cite{mw07}. Properties $(ii)$-$(iv)$ follow easily    
       from the definition of $v_{r}$ and the uniform convergence. Finally,  property $(v)$ follows from \eqref{e:nondegen}. 
     Then $v_0= r^2 g$ where the spherical piece $g$ satisfies $2ng + \Delta_{\theta} g =-1$. 
      We now utilize the Fredholm Alternative for existence for weak solutions, see Chapter 6 in \cite{e10}. 
     Since $f_1$  (in \eqref{expansion})  is a nontrivial solution, the solution $g$ exists if and only if 
     \[
     0= \langle -1 , h \rangle = \int_{\C \cap \partial B_1} -h ,
     \]
     for all $h$ such that $2n h + \Delta_{\theta} h =0$ (since the operator $2n + \Delta_{\theta}$ is self-adjoint). We recall that  
     $2nf_{1} + \Delta_{\theta} f_{1}=0$. 
     Then a necessary condition for $g$ to exist is that 
     \[
      0=\langle -1, f_1 \rangle = \int_{\C \cap \partial B_1} f_1.
     \]  
    Since $f_1>0$ in $\C \cap \partial B_1$, the above equality cannot be true.  Consequently, a solution $g$ cannot exist, and we have a contradiction.  
 \end{proof}

   %%%%%%%%%%%%%
 %%%%%%%%%%%%%
 
 \section{Boundary Harnack in Lipschitz domains}  \label{s:lipschitz}
 
 %%%%%%%%%%%%%
 %%%%%%%%%%%%%
 
 In this section we consider Lipschitz domains $\Omega_{L,R}$ where 
 \[
  \Omega_{L,R} := \{(x',x_n) \in B_R : x_n>g(x')\}, 
  \]
  and
 $g$ is Lipschitz with constant $L$, that is $|g(x')-g(y')|\leq L |x'-y'|$. We will assume $g(0)=0$, and will write $\Omega_{L}$ 
 when $R=1$ and $\Omega_{L,\infty}$ if $R=\infty$. Also, we define
 $u \in \mathcal{S}(\Omega_{L,R})$ if 
  \[
   \begin{aligned}
      \Delta u(x) &=0 \text{ in } \Omega_{L,R}  ,\\
      u(x)&=0 \text{ on } \Omega_{L,R}^c\cap B_R,
   \end{aligned}
  \]
 and for $\gamma \in \mathbb{R}$ we define
 $u \in \mathcal{S}(\Omega_{L,R},d^{\gamma})$ if 
  \[
   \begin{aligned}
     |\Delta u(x)| &\leq (\text{dist}(x,\partial \Omega_{L,R} \cap  B_R))^{\gamma} & \ &\text{ in } \Omega_{L,R} ,  \\
      u(x)&=0 & \ &\text{ on } \Omega_{L,R}^c\cap B_R.    
    \end{aligned}
  \]
 It will be necessary to use the solutions on right circular cones as barriers. Consequently, we define
  \[
    \C_L := \{(x',x_n): x_n > L |x'|\}.
   \]
 If $u \in \mathcal{S}(\C_{L,\infty})$ and $u\geq 0$, then as noted in Section \ref{s:cone}, we have  $u=r^{\alpha_1}f_1(\theta)$ 
 (and unique up to multiplicative constant), and we will denote $r^{\alpha_1}f_1$ by $u_L$. 
 We note that the definition of $\C_{-L}$ still makes sense when $-L<0$; although, the cone $\C_{-L}$ is not convex. 
 In this situation we write $\C_{-L}$ and similarly $u_{-L}$ for the nonnegative solution with zero boundary data on $\C_{-L}$.

% It will also be convenient to define for any function $v$ 
 %\[
%  S(v,r,x):= \sup_{B_r(x)} v. 
 %\]
 %For simplicity we write $S(v,r)$ when the ball $B_r$ is centered at the origin. 
 In order to prove a boundary Harnack principle with a right hand side for Lipschitz domains we will adapt the proof 
 from Section \ref{s:cone} in the following manner:
 \begin{itemize}
  \item We will employ compactness methods and thus need a convergence result provided by Lemma \ref{l:conv}.   
  \item We will need to bound the behavior of a nonnegative harmonic function at the boundary from above and below 
  which is given in Lemma \ref{l:contr}. 
  \item We will need a Liouville type result which is given in Lemma \ref{l:liousville}. 
    \item We will then adapt the proof of Theorem \ref{t:cone} (again using compactness techniques) to obtain the proof of 
  Theorem \ref{t:mainrhs}. 
 \end{itemize}

 \begin{lemma}   \label{l:control}
  Let $L \leq M$ and $u \in \mathcal{S}(\Omega_L)$ with $u \geq 0$.   
  Let $x \in \partial \Omega_L \cap B_{1/2}$ and let $y \in \Omega_L$ with dist$(y,\partial \Omega_L)>\delta$.   
  Then there exists a constant $C=C(n,M,\delta)$  such that 
   \[
    \begin{aligned}
    &\sup_{B_r(x)} u \leq C u(y) & \ & \text{ for all } r\leq 1/4 , \\
    &\sup_{B_{1/4}} u \geq C^{-1} u(y). 
    \end{aligned}
   \]
 \end{lemma}
 
 We give later a proof of a  more general version of this lemma; see Lemma \ref{l:3control} in Section \ref{s:div}.

 \begin{lemma}  \label{l:conv}
  Let $\Omega_{L_k,R_k}$ be a sequence of domains with $L_k \leq M$, $R_k \geq 1$, and $0 \in \partial \Omega_{L_k}$. Let 
  $u_k \in \mathcal{S}(\Omega_{L_k,R_k}, d^{\gamma})$, and let $\alpha$ be the degree of homogeneity for $u_M$, and assume
    $2-\alpha+\gamma >0$. 
  We further assume either       
  \[
    \begin{aligned}
      &(1) \quad u_k \geq 0  \ \text{ and }  \sup_{B_{1/2}} u_k \leq 1 , \quad or \\ 
      &(2)\ \sup_{B_r} u_k \leq C r^{\beta}  \ \text{ for } r\leq 1 \text{ and some constants } C, \beta >0.  
     \end{aligned}
   \]
  Then there exists a subsequence with a limiting domain $\Omega_{L_0,R_0}$ and a limiting function 
  $u_0 \in \mathcal{S}(\Omega_{L_0,R_0},d^{\gamma})$ such that  
  \[
   \sup_{B_r} |u_k - u_0| \to 0 \qquad \hbox{as } k \to \infty ,
  \]
  for all $r<R_0$. 
 \end{lemma}
 
 We give later a proof of the more general Lemma \ref{l:conv} in Section \ref{s:div} that implies Lemma \ref{l:conv}.

 \begin{lemma}  \label{l:contr}
  Let $u\in \mathcal{S}(\Omega_L)$ with $u\geq 0$.  Let $L<M$ and let 
  $\alpha_1$ be the degree of homogeneity for $u_{M}$ and $\beta$ the degree of homogeneity for $u_{-M}$. 
   There exists constants $c_1,c_2$ depending only on $n,M,$ and $M-L$ such that for any 
   $x \in \partial \Omega_L \cap B_{1/2}$
   \[
    \begin{aligned}
     &(1)\sup_{B_{r}(x)} u \geq c_1 u(e_n/2) r^{\alpha_1} , \\ 
     &(2)\sup_{B_{r}(x)} u \leq c_2 u(e_n/2) r^{\beta} ,
    \end{aligned}
   \]
   for any $r \leq 1/4$. 
 \end{lemma}
 
 \begin{proof}
  Let $x \in \partial \Omega_{L} \cap B_{1/2}$.  From Lemma \ref{l:control} we have 
  $u(y) \geq C_1 u(e_n/2)$ for any $y \in \partial B_{1/4}(x) \cap (\C_M+x)$.  We now consider the translated function $u_{M}(y-x)$ which is unique up to multiplicative 
  constant and homogeneous of degree $\alpha_1$. By multiplying $u_{M}(y-x)$ by a positive
  constant, we may assume that $\sup_{B_{1/4}(x)} u_{M}(y-x)=C_1 u(e_n/2)$. Then $u_M(y-x)\leq u(y)$ on $\partial B_{1/4}(x) \cap (\C_{M}+x)$. 
  Since $u_M(y-x)=0$ on $\partial (\C_M +x)$, then from the comparison principle we conclude that $u_M(y-x)\leq u(y)$ on $B_{1/4}(x) \cap (\C_{M}+x)$. 
  Then
  \[
   \sup_{B_r(x)} u(y) \geq \sup_{B_r(x)} u_{M}(y-x) = C_1 u(e_n/2) r^{\alpha_1},
  \]
  which proves (1). 
  In a similar manner we obtain property (2) by bounding $u$ from above by $u_{-M}$. 
 \end{proof}
 
 \begin{corollary}\label{c:contr} (To Lemma \ref{l:contr})
  Let $u\in \mathcal{S}(\Omega_L)$ with $L<M$, and let $\beta$ be 
  the degree of homogeneity for $u_{-M}$. 
   There exists a constant $C$ depending only on dimension $n,M,$ and $M-L$ such that for any 
   $x \in \partial \Omega_L \cap B_{1/2}$
   \[
     \sup_{B_{r}(x)} |u| \leq C (\sup_{B_{1}} |u|) r^{\beta} ,
   \]
   for any $r \leq 1/4$. 
 \end{corollary}
 
 \begin{proof}
  We consider the solution 
   \[
     \begin{aligned}
      \Delta v &=0 & &\text{ in } \Omega_L , \\
      v &= 0 & & \text{ on } \Omega_L^c \cap B_1 , \\
      v &= \sup_{B_{1}} u & &\text{ on } \partial B_{1} \cap \Omega_L.  
     \end{aligned}
   \] 
   From Lemma \ref{l:contr} we have that 
   \[
    v \leq C v(0,1/2) r^{\beta} \leq C (\sup_{B_1} v) r^{\beta}. 
   \]
   From the comparison principle we have that $u \leq v$ on $B_1$. By considering $-v$ we obtain a similar bound 
   from below to conclude the proof. 
 \end{proof}
 
 \begin{remark}   \label{r:rescale}
  A rescaling and translation to the origin of Corollary \ref{c:contr} 
  shows that if $L<M$ and $\beta$ is the degree of homogeneity for $u_{-M}$, then 
  if $u \in \mathcal{S}(\Omega_L,R)$ with $R>2$, then 
   \[
    (\sup_{B_1}|u|)R^{\beta} \leq c_2 \sup_{B_R} |u|.  
   \]
 \end{remark}

 \begin{lemma}   \label{l:liousville}
  Let $u,v \in \mathcal{S}(\Omega_{L,\infty})$ with $u,v \geq0$, then $u = cv$ for some 
  constant $c\geq0$. 
 \end{lemma}
 
 \begin{proof}
  Consider $w=u+v$. Then $w \geq u$. Let $c \geq 1$ be the largest constant such that $cu \leq w$. Then there exists a sequence $\{x_k \} \in \Omega_{L,\infty}$ such that 
   \begin{equation}  \label{e:lio1}
    \lim_{k \to \infty} \frac{c u(x_k)}{w(x_k)} = 1. 
   \end{equation}
   We now invoke the standard boundary Harnack principle for Lipschitz domains on the nonnegative harmonic functions $w-cu$ and $w$.  There exists $C_1>0$ such that 
   \begin{equation}   \label{e:lio2}
    \sup_{B_{r_k}} \frac{w-cu}{w} \leq C_1 \inf_{B_{r_k}} \frac{w-cu}{w} \leq C_1 \left(1-\frac{cu(x_k)}{w(x_k}\right), 
   \end{equation}
   with $r_k := \max \{2|x_k|, k\}$. From \eqref{e:lio1} the right hand side of \eqref{e:lio2} goes to zero. Then $w\equiv cu$, and the result follows.  
 \end{proof}

 \begin{lemma}  \label{l:lio2}
  Let  $u,v \in \mathcal{S}(\Omega_{L,\infty})$ with $u\geq 0$. If there exist  constants $C,R>0$ such that 
   \[
    \sup_{B_r} |v| \leq  C\sup_{B_r} u  \quad \text{ for } r\geq R,
   \]
   then $v = cu$ for some $c \in \R$. 
 \end{lemma}

 \begin{remark}
  The significance of Lemma \ref{l:lio2} is that we do not require $v \geq 0$. 
 \end{remark}
 
 \begin{proof}
  Let $v_r$ satisfy 
  \[
   \begin{aligned}
    \Delta v_r &= 0 \quad \text{ in } B_r \cap \Omega_{L,\infty} \\
    v_r &= v^+ \ \text{ on } \partial (B_r \cap \Omega_{L,\infty}).  
   \end{aligned}
  \]
  By Lemma \ref{l:contr} we have that 
  \[
   \sup_{B_{2r}} |v| \leq C\sup_{B_{r}} |v| \text{ for } r \geq 1,
  \] 
  and some constant $C$ independent of $r$. Then 
  \[
   v_r(r/2,0,\ldots,0) \leq C \sup_{B_r} |v| \leq C \sup_{B_{r/2}} |v| \leq C \sup_{B_{r/2}} u
   \leq C u(r/2,0,\ldots,0)  
   \]
   with the last inequality following form Lemma \ref{l:control}. By the standard boundary Harnack principle 
   \[
    \sup_{B_{r/2}} \frac{v^+}{u} \leq \sup_{B_{r/2}} \frac{v_r}{u} \leq C \frac{v(r/2,0,\ldots,0)}{u(r/2,0,\ldots,0)}\leq C_1. 
   \]
   Since the constants are independent of $r \geq 1$, we conclude that $v^+ \leq C_1 u$ in $\Omega_{L,\infty}$. 
   The same argument holds for $v^-$ so that $|v| \leq C_1 u$ in $\Omega_{L,\infty}$.

  Let $w=C_1u-v\geq0$. Then from Lemma \ref{l:liousville} we have that $w=cu$ for some constant $c$, so that
  $v=(C_1-c)u$.   
 \end{proof}
 
 We will need an improvement over the previous lemma. 
  \begin{lemma}   \label{l:logstop}
   Let  $u,v \in \mathcal{S}(\Omega_{L,\infty})$ with $u\geq 0$. If there exist  constants $C,R>0$ such that 
   \[
    \sup_{B_r} |v| \leq  C(\ln(r+1))\sup_{B_r} u  \quad \text{ for } r\geq 1,
   \]
   then $v = cu$ for some $c \in \R$. 
  \end{lemma}
 
 In Section \ref{s:div} we give a  proof of a more general result in Theorem \ref{t:logstop}.

 %%%%%%%%%%%%%
 %
%       Theorem: Lipschitz  domains 
%
  %%%%%%%%%%%%%

    \begin{theorem}   \label{t:mainrhs}
      Let $0 \in \partial \Omega_{L,2}$ with $L<M$, and fix $x^0 \in \Omega_L$. Assume further that 
      $B_1 \cap \{x_n > 1/4\} \subseteq \Omega_{L,2}$. Let $\alpha_1$ be the degree of 
      homogeneity for $u_M$. 
      Let $u,v \geq 0$ and $u,v \in \mathcal{S}(\Omega_{L,2},d^{\gamma})$ with 
      $\Delta u, \Delta v \leq 0$ and $u(x^0)=v(x^0)=1$, and assume that $2-\alpha_1+\gamma>0$. 
      Then there exists a uniform constant $C>0$ 
      (depending only on dimension $n$, 
      Lipschitz constant $M$, $M-L$, and $\text{dist}(x^0,\partial \Omega_{L,2})$) 
      such that 
       \begin{equation}  \label{e:ine}
        C^{-1} v(x) \leq u(x) \leq Cv(x) ,
       \end{equation}
       for all $x \in B_{1/2}$. 
    \end{theorem}

   \begin{proof}
    It will suffice to assume $\Delta u=0$ in $\Omega_{L,2}$ and show that $v \leq Cu$ in $B_{1/2}$. 
    We initially prove the theorem for a fixed $x^0 \in \C_M \cap \partial B_{1/2}$.

    Suppose by way of contradiction that the theorem is not true. 
    Then there exists $u_k \in \mathcal{S}(\Omega_{L_k,2})$ and $v_k \in \mathcal{S}(\Omega_{L_k,2}, d^{\gamma})$
    with $\Delta v_k \leq 0$ and $v_k(x^0)=u_k(x^0)=1$ and points $x_k \in \Omega_{L_k,2}\cap B_{1/2}$ such that 
       \begin{equation}  \label{e:explode}
     ku_k(x_k) < v_k(x_k). 
    \end{equation}
    Harnack chains work on Lipschitz domains; therefore, from the interior Harnack inequality,  by multiplying $v_k$ and $u_k$ by constants (uniformly bounded above and below), 
    we may assume that 
    \begin{equation}  \label{e:lateral}
    v_k(x_k',e_n)=u_k(x_k',e_n)=1. 
    \end{equation}
    Because of the interior Harnack inequality, in order for \eqref{e:explode} to occur, 
    it is necessary that dist$(x_k, \partial \Omega_{L_k}) \to 0$. We let $y_k = (x_k', e_n)$, and
    similar to the proof of Theorem \ref{t:cone} we define 
     \[
   w_{r,k}(x) := \frac{v_k(rx)-\frac{v_k(r y_k)}{u_k(r y_k)}u_k(rx)}{
    \sup_{B_1 \cap \mathcal{C}} |v_k(rx)-\frac{v_k(ry_k)}{u_k(ry_k)}u_k(rx)|},
   \]
   and 
    \[
    a_{k,m} := \sup_{B_1 \cap \mathcal{C}} |v_k(2^{-m}x)-\frac{v_k(2^{-m}y_k)}{u_k(2^{-m}y_k)}u_k(2^{-m}x)|
        \sup_{B_{2^{-m}}} u_k. 
    \]
    For fixed $m$, it follows by the standard boundary Harnack Principle that 
    \[
     b_m := \sup_k a_{k,m} < \infty. 
    \]
    Just as in the proof of Theorem \ref{t:cone}, by the assumption \eqref{e:explode} we necessarily obtain that
    \[
     \sum b_m = \infty. 
    \]
    From Lemma \ref{l:sumdiv} there exists a subsequence $b_{m_l}$ such that for any fixed $j$ we have
    \[
     \limsup_{m_l \to \infty} \frac{\sum_{i=1}^j b_{m_l -i}}{b_{m_l}}  \leq j. 
    \]
    For any $N>0$ with $N \in \mathbb{N}$, there exists $\tilde{N} \in \mathbb{N}$, such that if $m_l\geq \tilde{N}$, then there is a $k=k(m_l)$ such that $a_{k,m_l}$ satisfies 
    \begin{equation} \label{e:kml}
     \frac{\sum_{i=1}^j a_{k,m_l -i}}{a_{m_l}}  \leq C j \text{ for } j\leq N. 
    \end{equation}  
    If we let $r_k = 2^{-m_l}$, then for those chosen $k$ we have that 
   \[
    \sup_{B_{2^j}} |w_{r_k,k}(x)| 
    = \frac{\sup_{B_{2^j}} |v_k(r_k x) - \frac{v( r_k y_k)}{u( r_k y_k)}u(r_k x)|}{
       \sup_{B_{r_k}} u_k},          
   \]
    and 
     \[
   \begin{aligned}
   & \left| \frac{v_k(2^{j-i} r_k y_k)}{u_k(2^{j-i} r_k y_k)}u_k(2^{j-i-1}r_k y_k) - \frac{v_k(2^{j-i-1} r_k y_k)}{u_k(2^{j-i-1} r_k y_k)}u_k(2^{j-i-1}r_k y_k) \right| \\
   &\  = \left| \frac{v_k(2^{j-i} r_k y_k)}{u_k(2^{j-i} r_k y_k)}u_k(r_k 2^{j-i-1}y_k) - v(2^{j-i-1} r_k y_k) \right|  \\
   & \leq a_{k,m_l-(j-i)}  \sup_{B_{2^{j-i} r_k}} u_k. 
    \end{aligned}
   \]
   Also we have that 
    \begin{equation}  \label{e:seq3l}
   \begin{aligned}
    &\sup_{B_{2^j}} \left| \frac{v_k(2^{j-i} r_k y_k)}{u_k(2^{j-i} r_k y_k)}u_k(r_k x) - \frac{v_k(2^{j-i-1} r_k y_k)}{u_k(2^{j-i-1} r_k y_k)}u(r_k x) \right| \\
    &=  \sup_{B_{2^j}} \frac{u_k(r_k x)}{u_k(2^{j-i-1} r_k y_k)}  \\
    & \quad \times 
    \left| \frac{v_k(2^{j-i} r_k y_k)}{u_k(2^{j-i} r_k y_k)}u_k(2^{j-i-1}r_k y_k) - \frac{v_k(2^{j-i-1} r_k y_k)}{u_k(2^{j-i-1} r_k y_k)}u_k(2^{j-i-1}r_k y_k) \right|  \\   
    & \leq \sup_{B_{2^j}}\frac{u_k(2^j r_k x)}{u_k(2^{j-i-1} r_k y_k)} a_{m_l-(j-i)} \sup_{B_{2^{j-i} r_k}} u_k  \\
    & \leq C a_{m_l-(j-i)} \frac{u_k(2^j r_k y_k)}{u_k(2^{j-i-1} r_k y_k)} u_k(2^{j-i}r_k) u_k(y_k) \\
    & \leq C a_{m_l-(j-i)} u_k(2^j r_k y_k),
   \end{aligned}
   \end{equation} 
   with the last inequality following from the bounds in Lemma \ref{l:contr}. Then combining estimates \eqref{e:kml}
   and \eqref{e:seq3l} we obtain 
   \begin{equation}  \label{e:loggrowth}
   \sup_{B_{2^j}} |w_{r_k,k}(x)| \leq C j \frac{\sup_{B_{2^j}} u_k(r_k x)}{\sup_{B_1} u_k(r_k x)}.  
   \end{equation}
    
    We now use Lemma \ref{l:conv} as $r_k \to 0$ to obtain limiting functions and domains. 
   By  choosing a subsequence we first consider the limit function 
   \[
    u = \lim_{r_k \to 0} \frac{u_k(r_k x)}{\sup_{B_{r_k}} u_k}.
   \] 
  We also obtain a limiting global Lipschitz domain $\Omega$ on which $u$ is the unique (up to multiplicative constant) nonnegative harmonic function that vanishes on the boundary. A further subsequence guarantees $y_k \to y_0 \in \Omega$. 
   As in the proof of Theorem \ref{t:cone}, as $r_k \to 0$ we have that $|\Delta w_{r_k,k}| \to 0$. Then 
   by picking a further subsequence, as $r_k \to 0$ we obtain a limiting global Lipschitz domain function $w$ with the following properties
    \[
      \begin{aligned}
       &(i) \quad \Delta w =0 \text{ in } \Omega \text{ and } w=0 \text{ on }  \Omega^c  ,\\
       &(ii) \quad \sup_{B_1} |w| =1&\  ,\\
       &(iii) \quad w(y^0) =0  ,  \\
       &(iv)  \quad w(x) \leq C u(x) \ln(|x|+1) \ \text{ for } |x| \geq 1\quad \text{ from }   \eqref{e:loggrowth}.   
       \end{aligned}
     \] 
    From property $(iv)$ Lemma \ref{l:logstop} guarantees $w(x)=cu(x)$, but then both properties $(ii)$ and $(iii)$
    cannot hold since $u(y^0)>0$. 
    
     An interior Harnack inequality with a Harnack chain will also give the result for 
     $x^0 \in B_{1/2}\cap \Omega_L$ and the constant $C$ depending on dist$(x^0,\partial \Omega_L)$. 
     \end{proof}

 \begin{remark}
  If one does not assume that $B_1 \cap \{x_n > 1/4\} \subseteq \Omega_L$, then one may  modify the proof
  at the expense that \eqref{e:ine} holds for $x \in B_r$ with $r \leq \min\{(2L)^{-1},1/2\}$. 
 \end{remark}
 
 \begin{corollary}    \label{c:mainrhs}
  Let $u,v,\Omega$ have the same assumptions as in Theorem \ref{t:mainrhs} with the exception that $v$ is no longer required to satisfy $\Delta v\leq 0$ and $v\geq0$. 
  Assume that $\gamma=0$ and $M$ is small enough such that $2-\alpha_1>0$. Then there exists $\beta>0$ depending on $M$ and $M-L$ such that 
   \begin{equation}  \label{e:hest}
    \left\| \frac{v}{u} \right\|_{C^{0,\beta}(B_{1/2}\cap \Omega_L)}\leq C \frac{\left (\| v \|_{L^{\infty}(\Omega_L)} + \| f \|_{L^{\infty}(\Omega_L)}\right)}{u(re_n/2)}.
   \end{equation}
 \end{corollary}
 
 \begin{proof}
  The argument for how the boundary Harnack principle implies H\"older regularity is now standard (see \cite{ath-caff-1985}). We outline how to 
  adapt to the case when $w$ solves 
   \[
    \begin{aligned}
    \Delta w &= -1 & \ &\text{ in } \Omega_L, \\
    w&=0 & \ &\text{ on } \partial \Omega_L \cap B_1, \\
    w&=|v|  & \ &\text{ on }  \Omega_L \cap \partial B_1.
    \end{aligned}
   \]
   From \eqref{e:ine} it follows that 
    \begin{equation}  \label{e:ine2}
     \left\| \frac{w}{u} \right\|_{\Omega_{L,1/2}} \leq C \frac{w(e_n/2)}{u(e_n/2)}.  
    \end{equation}   
    It is now a standard argument (see \cite{ath-caff-1985}) that \eqref{e:ine2} implies that there exists $\beta$ depending on $M$ and $M-L$  such that 
    \[
     \left\| \frac{w}{u}\right\|_{C^{0,\beta}(\Omega_{L,1/2})} \leq C \frac{w(e_n/2)}{u(e_n/2)}. 
    \]
    Since $|v|\leq w$ we obtain a H\"older growth bound for $v/w$ at $\partial \Omega_{L,1/2}$. The interior H\"older estimates for both $v$ and $u$ combined with the fact 
    that $|v|/u$ is bounded give interior H\"older estimates for $v/u$. The interior H\"older estimates combined with the boundary H\"older estimates for $v/u$ are combined in 
    a standard way to conclude \eqref{e:hest}. 
  \end{proof}

%%%%%%%%%%%%%%%%%% 
  %%%%%%%%%%%%%%%%%% 
%
%          Elliptic operators
 %
  %%%%%%%%%%%%%%%%%% 
 %%%%%%%%%%%%%%%%%% 

 \section{Second-Order Elliptic Operators}  \label{s:div}

 The techniques employed in Section \ref{s:lipschitz} can be applied to other elliptic operators, and in this section we extend the results of 
 Section \ref{s:lipschitz} to second order linear elliptic operators in divergence form on Lipschitz domains.  Specifically, we consider operators of the form
 \[
  \mathcal{L} u = (a^{ij}u_i)_j + b^i u_i + cu
 \]
 %r
 %\[
 % \mathcal{L} u = a^{ij}u_{ij} + b^i u_i + cu. 
 %\]
 We assume the following ellipticity conditions  
  \[
    \Lambda^{-1} |\xi|^2 \leq \langle a^{ij}(x) \xi, \xi \rangle \leq \Lambda |\xi|^2, 
  \]
 for some $\Lambda >0$ and for all nonzero $\xi \in \R^n$. Furthermore, $a^{ij}(x)$ is a real $n \times n$ matrix. 
 %In the nondivergence setting we also assume the coefficients $a^{ij}$ to be symmetric. 
 For the lower order terms we assume $|c(x)|, \sum |b^i(x)| \leq \Lambda-1$ and that $c(x) \leq 0$. 
  
 We will continue with the notation of Section \ref{s:lipschitz}; however,
 we now write 
 $u \in \mathcal{S_{\mathcal{L}}}(\Omega_{L,R})$ if 
  \[
   \begin{aligned}
     \mathcal{L} u(x) &=0 \text{ in } \Omega_{L,R} \cap B_R ,\\
      u(x)&=0 \text{ on } \Omega_{L,R}^c\cap B_R,
   \end{aligned}
  \]
 and that
 $u \in \mathcal{S_{\mathcal{L}}}(\Omega_{L,R},d^{\gamma})$ if 
  \[
   \begin{aligned}
     |\mathcal{L} u(x)| &\leq (\text{dist}(x,\partial \Omega_{L,R} \cap  B_R))^{\gamma} \ \text{ in } \Omega_{L,R} , \\
      u(x)&=0 \ \text{ on } \Omega_{L,R}^c\cap B_R.
   \end{aligned}
  \]
 To apply the H\"older continuity estimates for elliptic operators we will require that   
  $$\gamma>-2/n ;$$ 
 see Theorem 8.27 in \cite{gt01}.
 Since the boundary is Lipschitz, this will ensure the correct integrability 
 assumptions for the right hand side. We will assume these bounds throughout the section whenever referencing $\mathcal{S_{\mathcal{L}}}(\Omega_{L,R},d^{\gamma})$.

 From the forthcoming Lemma \ref{l:genliousville},  it will follow that if $u \geq 0$ and $u \in \mathcal{S_{\mathcal{L}}}(\C_L,\infty)$, then $u$ is unique up to multiplicative constant and we again denote $u$ by 
  $u_L$; however, $u_L$ will not necessarily be homogeneous.  We recall that $\C_L$ is defined although not convex when $L<0$. To emphasize 
  when $-L<0$, we again write $\C_{-L}$ and $u_{-L}$ when 
  $u \in \mathcal{S_{\mathcal{L}}}(\C_{-L})$. We will follow the same outline as in Section \ref{s:lipschitz}. 
  
  In Section \ref{s:lipschitz} we utilized the standard boundary Harnack principle. Since the standard boundary Harnack principle is unavailable when considering the zero-order
  term $c(x)$, we prove the next two Lemmas under the situation $b^i, c \equiv 0$.  
  
  \begin{lemma}   \label{l:genliousville}
    Let $u,v \in \mathcal{S}_{\mathcal{L}(\Omega_{L,\infty})}$ with $u \geq 0$. Assume also that $b^i, c \equiv 0$. If there exists $C>0$ such that 
    \[
     \sup_{B_R} |v| \leq C \sup_{B_R} u \text{ for } R \geq 1, 
    \]
    then there exists $c \in \mathbb{R}$ such that $v \equiv cu$. 
  \end{lemma}
 
 \begin{proof}
  When $b^i, c \equiv 0$, there is a standard Boundary Harnack principle for divergence form equations \cite{cfms81};
  %and for nondivergence form equations \cite{b84}
   therefore, 
  the proof of Lemma \ref{l:liousville} holds in this situation, and so the proof of Lemma \ref{l:lio2} also holds as well. 
 \end{proof}
 
  \begin{theorem}   \label{t:logstop}  
  Assume $\mathcal{L}$ has no lower order terms; i.e,  $\mathcal{L}w = \partial_j(a^{ij} u_i)$. 
  Let $v,u \in \mathcal{S}_{\mathcal{L}}(\Omega_{L,\infty})$ with $u \geq 0$, and assume $b^i, c \equiv 0$. If there exists $C>0$ such that 
  \[
   |v(x)| \leq C\ln(|x|+1) u(x) \text{ for } |x| \geq 1,
  \]
  then $v(x)=cu(x)$ for some $c \in \mathbb{R}$. 
 \end{theorem}
 
 \begin{remark}
  The  proof given below for Theorem \ref{t:logstop} depends on the function $g$ being \textit{slowly varying} at $\infty$. Thus, the same proof 
  will actually show a stronger result: If for every $\epsilon>0$ there exists $C_{\epsilon}$ such that 
  \[
   |v(x)|\leq C_{\epsilon} |x|^{\epsilon}u(x) \text{ for } |x| \geq 1, 
  \]
  then $v=cu$ for some $c \in \mathbb{R}$. 
 \end{remark}
 
 \begin{proof}
  Suppose by way of contradiction that $v$ is not a constant multiple of $u$. Then by Lemma \ref{l:genliousville} we have 
  \[
  \limsup_{R \to \infty} \left( \sup_{B_R} \frac{|v(x)|}{u(x)}  \right)= \infty. 
  \]
  If $h(R)=\sup_{B_R} |v|/u$, let $g$ be the concave envelope of $h$. By assumption $g(R) \leq C \ln(R+1)$ for $R \geq 1$.
  The function $g$ satisfies, 
  \begin{equation}  \label{e:slowlyvarying}
   \lim_{R \to \infty} \frac{g(CR)}{g(R)}=1 \text{ for any } C>0. 
  \end{equation}
  There also exist infinitely many $R_k$ such that $h(R_k) = g(R_k)$. We define
  \[
   u_k(x):=\frac{u(R_k x)}{\sup_{B_{R_k}} u}  \ \text{ and } \ v_k(x)=\frac{v(R_k x)}{\sup_{B_{R_k}} |v|}.
  \]
  From \eqref{e:slowlyvarying} we have that for any fixed $\rho>1$, that 
  \begin{equation}  \label{e:slowlyvarying2}
    \lim_{R_k \to \infty} \frac{\sup_{B_{\rho}} |v|}{ \sup_{B_{\rho}} u} \leq 1. 
  \end{equation}
  By uniform continuity estimates up to the boundary for divergence form equations, by picking a subsequence 
  we have that 
  \[
   \lim_{R_k \to \infty} u_k(x) = w_1, 
  \]
  and that $u_k$ converges uniformly to $w_1$ on compact sets. From \eqref{e:slowlyvarying2} we also have that 
  \[
   \lim_{R_k \to \infty} v_k(x)= w_2, 
  \]
  with uniform convergence on compact sets and with $|w_2| \leq C w_1$. There also exists a limiting operator $\mathcal{L}_0$ and limiting Lipschitz domain $\tilde{\Omega}$
  such that 
  $w_1, w_2 \in \mathcal{S}_{\mathcal{L}_0}(\tilde{\Omega}_{L, \infty})$. Then from Lemma \ref{l:genliousville} we conclude that $w_2 = c w_1$ for some $c \in \mathbb{R}$. 
  Without loss of generality we assume $c=1$. 
  
  Then for any $\epsilon>0$, there exists $N_{\epsilon} \in \mathbb{N}$ such that if $k \geq N_{\epsilon}$, then
  \[
   \sup_{B_2} |v_k - u_k| < \epsilon. 
  \]
  By the standard boundary Harnack principle 
  \[
   \sup_{B_1} \frac{|v_k - u_k|}{u_k} \leq C \frac{\| v_k - u_k \|_{L^{\infty}(B_2)}}{u_k(e_n/2)} \leq C \epsilon. 
  \] 
 We fix $x^1 \in \Omega_{L,\infty}$. We also have
 \[
  \lim_{R_k \to \infty} \frac{|v_k(R_k^{-1}x^1)|}{|u_k(R_k^{-1} x^1)|} \leq \frac{v(x^1)}{u(x^1)}\frac{1}{g(R_k)} \to 0. 
 \]
 Then for large enough $k$ we have  $|v_k(R_k^{-1} x^1)|\leq u_k(R_k^{-1} x^1)/2$. Finally, we conclude then that  
 
 \[
  \frac{u(x^1)/2}{u(x^1)} \leq \sup_{B_1} \frac{|v_k - u_k|}{u_k}  \leq C \epsilon.
 \]
 The $C$ in the above estimate only depends on the ellipticity constants of $a^{ij}$ and the Lipschitz constant for the 
 domain. Consequently, we may choose $\epsilon$ small enough so that $2C \epsilon<1$, which implies $u(x^1)=0$
 which is a contradiction since $u>0$ in $\Omega_{L,\infty}$.  
 \end{proof}

 %We will also need a global growth rate as well as a nondegeneracy estimate. 
%
 %\begin{lemma}   \label{l:globabove}
%  Assume $b^i, c \equiv 0$. There exists $\alpha, c, C>0$ and depending only on $\Lambda$ and $M$ such that 
%  \[
%   u_M(x) \leq Cu_M(0,1/2)|x|^{\alpha}  \text{  for }  |x| \geq 1.   
%   u_M(x) \geq cu_M(0,1/2)|x|^{\alpha}  \text{ for }  |x| \leq 1. 
%  \]
% \end{lemma}
% 
% \begin{proof}
% If $\mathcal{L}$ is of nondivergence form, this result was proven in \cite{m67}. 
%   If $\mathcal{L}$ is in divergence form, one may flatten the boundary and obtain a new divergence form $\mathcal{L}$ and the 
%   ellipticity coefficients change depending on the Lipschitz constant of the original domain. Thus, we may assume $\Omega$ is the upper half plane. 
%    Using a Harnack chain of balls of increasing size we conclude that 
%  \[
%   u(2^k e_n) \leq  C^k u(e_n/2), 
%  \]
%  where $C=C(n,\lambda)$ is the constant from the Harnack inequality for divergence form elliptic solutions. 
%   Now $C=2^{\alpha}$, so that 
%  \[
%   u(2^k e_n) \leq  2^{\alpha k}.
%  \]
%  Now Lemma \ref{l:3control} does not depend on the results of this Lemma. A rescaled version of Lemma \ref{l:3control} gives 
%  \[
%   u(x) \leq C_1 u(e_n/2) |x|^{\alpha}. 
%  \]
 % A rescaling of the above result, then gives the estimate from below. 
% \end{proof}

 For the remainder of the section we no longer assume that the lower order terms are zero. 
 \begin{lemma}   \label{l:3control}
  Let $L \leq M$ and $u \in \mathcal{S_{\mathcal{L}}}(\Omega_L, d^{\gamma})$ with $u \geq 0$ and $\mathcal{L} u \leq 0$.  Assume also $0 \in \partial \Omega_L$.   
  Let $x \in \partial \Omega_L \cap B_{1/2}$ and let $y \in \Omega_L \cap B_r(x)$ with $r \leq 1/4$.  
  Then there exists a constant $C$ depending only on dimension $n$, $M$, and $\text{dist}(y, \partial \Omega_L)$ such that 
   \[
    \begin{aligned}
    &\sup_{B_r(x)} u \leq C u(y) & \ & \text{ for all } r\leq 1/4,  \\
    &\sup_{B_{1/4}} u \geq C^{-1} u(y). 
    \end{aligned}
   \]
 \end{lemma}
 
 \begin{proof}
   Since $u \geq 0$ and $\mathcal{L} u \leq 0$ and $u \in \mathcal{S_{\mathcal{L}}}(\Omega_{L,R},d^{\gamma})$, this is an application of 
   the interior weak Harnack inequality as well as uniform H\"older continuity up to the boundary, see \cite{gt01}. 
  
 \end{proof}

We now state the analogue of Lemma \ref{l:conv}. 
   \begin{lemma}  \label{l:3conv}
  Let $\Omega_{L_k,R_k}$ be a sequence of domains with $L_k \leq M$, $R_k \geq 1$, and $0 \in \partial \Omega_{L_k}$. Let 
  $u_k \in \mathcal{S_{\mathcal{L}}}(\Omega_{L_k,R_k}, d^{\gamma})$, and assume $\gamma >-2/n$.   Further assume either    
  \[
    \begin{aligned}
      &(1) \qquad u_k \geq 0 \text{ and } \sup_{B_{1/2}} u_k \leq 1,  \text{ or } \\ 
        &(2)\qquad \sup_{B_r} |u_k| \leq C r^{\beta} \text{ for } r\leq 1 \text{ and some constants } C, \beta >0.  
     \end{aligned}
   \]
  Then there exists a subsequence with a limiting domain $\Omega_{L_0,R_0}$ and a limiting function 
  $u_0 \in \mathcal{S}(\Omega_{L_0,R_0},d^{\gamma})$ such that  
  \[
   \sup_{B_r} |u_k - u_0| \to 0 ,
  \]
  for all $r<R_0$. 
 \end{lemma}
 
 \begin{proof}
  This is an application of uniform H\"older continuity up to the boundary, see \cite{gt01}. 
  \end{proof}

 \begin{lemma}  \label{l:3contr}
  Let $u\in \mathcal{S}_{\mathcal{L}}(\Omega_L)$ with $u\geq 0$ and $\mathcal{L} u \leq 0$.  Let $L<M$.  
     There exist constants $c_1,c_2, \beta,\alpha$ depending only on $n,M,\Lambda$ and $M-L$ such that for any 
   $x \in \partial \Omega_L \cap B_{1/2}$
   \[
    \begin{aligned}
     &(1) \qquad \sup_{B_{r}(x)} u \geq c_1 u(e_n/2) r^{\alpha} , \\ 
     &(2) \qquad \sup_{B_{r}(x)} u \leq c_2 u(e_n/2) r^{\beta}, 
    \end{aligned}
   \]
   for any $r \leq 1/4$. 
 \end{lemma}
 
 \begin{proof}
  Property $(2)$ is just uniform H\"older continuity up to the boundary. 
  Assume by way of contradiction that $(1)$ is not true. Then there exists a sequence satisfying 
  \[
   \begin{aligned}
    (1)&\qquad  u_k \in \mathcal{S}_{\mathcal{L}_k}(\Omega_{L_k}, d^{\gamma}) \\
    (2)&\qquad u_k \geq 0 \\
    (3)&\qquad \mathcal{L}_k u_k \leq 0 \\
    (4)&\qquad u_k(e_n/4) \leq \frac{u_k(e_n/2)}{k} \quad \text{ from Lemma } \ref{l:3control}. 
   \end{aligned}
  \]
  We rescale and let 
  \[
   w_k = \frac{u_k(r_k x)}{ \sup_{B_{r_k}} u_k}. 
  \]
  From Lemma \ref{l:3conv} we have that $w_k \to w_0$ uniformly and there is a limiting Lipschitz domain $\Omega_{L_0}$ 
  and limiting elliptic operator $\mathcal{L}_0$ such that 
  $u \in \mathcal{S}_{\mathcal{L}_0}(\Omega_{L_0})$. Now $w_0 \geq 0$, and from the definition of $w_k$ we conclude that $w_0$ is not identically zero.  
  However, $w_0(e_n/4)=0 $ which contradicts the strong maximum principle. 
 \end{proof}

 %\begin{remark}   \label{r:rescale}
 % A rescaling and translation to the origin of Corollary \ref{c:contr} 
 % shows that if $L<M$  then 
 % if $u \in \mathcal{S}(\Omega_L,R)$ with $R>2$, then 
 %  \[
 %   (\sup_{B_1}|u|)R^{\beta} \leq c_2 \sup_{B_R} |u|.  
 %  \]
% \end{remark}

  With the previous result,  the proof of Theorem \ref{t:genrhs} proceeds exactly as in the case of  Theorem \ref{t:mainrhs}.

    \begin{theorem}   \label{t:genrhs}
      Let $0 \in \partial \Omega_L$ with $L<M$, and fix $x^0 \in \Omega_L$. Assume further that 
      $B_1 \cap \{x_n > 1/4\} \subseteq \Omega_L$. Let $\gamma>-2/n$ and let $\alpha$ be as in Lemma \ref{l:3contr}.    
         Assume $u,v \geq 0$ and $u,v \in \mathcal{S}(\Omega_{L},d^{\gamma})$ with 
      $\mathcal L u, \mathcal L v \leq 0$ and $u(x^0)=v(x^0)=1$, and also  assume that $2-\alpha+\gamma>0$ with $\alpha$ as given in Lemma \ref{l:3contr}. 
      Then there exists a uniform constant $C>0$ 
      (depending only on dimension $n$, 
      Lipschitz constant $M$, $M-L$, and dist$(x^0,\partial \Omega_L)$) 
      such that 
       \begin{equation}  \label{e:ine3}
        C^{-1} v(x) \leq u(x) \leq Cv(x)
       \end{equation}
       for all $x \in B_{1/2}$. 
    \end{theorem}

    \begin{proof}
      The proof proceeds exactly as the proof of Theorem \ref{t:mainrhs}. We only highlight how the lower order terms vanish in the blow-up regime. 
       The rescaled functions 
      \[
        w_{r,k}(x) := \frac{v_k(rx)-\frac{v_k(r y_k)}{u_k(r y_k)}u_k(rx)}{
    \sup_{B_1 \cap \mathcal{C}} |v_k(rx)-\frac{v_k(ry_k)}{u_k(ry_k)}u_k(rx)|},
      \]
  satisfy 
  \begin{equation}  \label{e:lastdiv}
  \mathcal{L}_r w_{r,k} = \frac{r^2\mathcal{L} v_k(rx)}{
    \sup_{B_1 \cap \mathcal{C}} |v_k(rx)-\frac{v_k(ry_k)}{u_k(ry_k)}u_k(rx)|},
  \end{equation}
   where for a function $f$ we have 
      \[
       \mathcal{L}_r f = (a^{ij}(rx)f_i)_j + rb^{i}(rx)f_i + r^2c(rx)f. 
      \]
      Thus, in the blow-up regime the lower order terms disappear.
      As in the proof of Theorems \ref{t:cone} and  \ref{t:mainrhs} we may bound the denominator in \eqref{e:lastdiv} from below by $r_k^{\alpha} [\ln(1/r_k)]^2$ for 
        the constructed sequence of $r_k \to 0$. Using that the numerator is bounded from above by $r_k^{2-\gamma}$, we have that in the blow-up regime the right hand side
       and lower order terms vanish.  We then apply Lemma \ref{l:genliousville} and Theorem \ref{t:logstop} as in the proof of Theorem \ref{t:mainrhs}. 
    \end{proof} 

\appendix

\section{ \ }
\begin{lemma}   \label{l:sumdiv}
 Let $\{a_k\}$ be a sequence with $a_k \geq 0$,  and with infinitely many terms nonzero. Further assume that $\sum_{k=1}^{\infty} a_k = +\infty$. 
 Then there exists a subsequence $a_{k_l}$ such that for any $j \in \mathbb{N}$, 
 \begin{equation} \label{e:sumdiv}
  \limsup_{k_l \to \infty} \frac{\sum_{i=1}^j a_{k_l -i}}{a_{k_l}}  \leq j. 
 \end{equation}
\end{lemma}

\begin{proof}
 If $\limsup a_k = \infty$, one may simply choose a subsequence $a_{k_l}$ such that $a_{k_l -i} \leq a_{k_l}$ for any $0<i<k_l$ and the result immediately follows. 
 If $0<\limsup a_k<\infty$, one may choose a subsequence $a_{k_l}$ such that $\lim a_{k_l}=\limsup a_k$, and the result also follows. 
 
 We now consider the most difficult case when $\limsup a_k =0$. Define $f_1(k)=a_k$ for $k \in \mathbb{N}$, and interpolate linearly between integers for any value $x \geq 1$. Note that $\lim_{x \to \infty} f_1(x)=0$.  
 Let $k_1 = \max \{j \in \mathbb{N} \mid a_j = \max \{a_k\}\}$. We inductively choose 
 \[
 k_{l+1}=\max \{j > k_l \mid a_j = \max\{a_k\}_{k=k_l +1}^{\infty}\}.
 \]
  We define $f_2(k_l)=a_{k_l}$ and interpolate linearly in between values of the subsequence $\{k_l\}$. Then $f_2(x)$
 is strictly decreasing, and $f_2(k) \geq a_k$ for $k \geq k_1$. 
 
 We now use an inductive procedure to construct $f_3(x)$ which will be strictly decreasing and convex.  We choose a further subsequence by letting 
 $k_{l_1} = k_1$. If $k_{l_i}$ has been chosen, then we define $f_3$  by letting $f_3(k_{l_j}) =a_{k_{l_j}}$ for $1\leq j \leq i$ and linearly 
 interpolating for all other values $k_{l_1}\leq x \leq k_{l_i}$. We choose  
 \[
  k_{l_{i+1}} = \min \left\{k_j \mid k_j > k_{l_i} \text{ and } \frac{f_3(k_{l_i}-1)-f_3(k_{l_i})}{2} + \frac{f_2(k_j)-f_2(k_{l_i})}{2(k_j-k_{l_i})} \geq 0 \right\}.
 \]
 Such a minimum will exist since $f_3(x)$ is strictly decreasing on $k_{l_1}\leq x \leq k_{l_i}$. Then $f_1(x)\leq f_2(x)\leq f_3(x)$, and $f_3(x)$ is strictly decreasing, convex, and 
 $f_3(k_{l_i}) = a_{k_{l_i}}$. 
  For convenience we relabel $f_3(x)=g(x)$ and relabel our subsequence $k_{l_i}$ to be $k_{l}$.

Since $g > 0$, decreasing, and convex we may take a smooth approximation $g_{\epsilon}(x)$ such that 
$g_{\epsilon} \geq 0, g_{\epsilon}' \leq 0, $ and $g_{\epsilon}''\geq0$ it follows that 
$g_{\epsilon}'(x+1)g_{\epsilon}(x)\leq g_{\epsilon}'(x)g_{\epsilon}(x)\leq g_{\epsilon}'(x)g_{\epsilon}(x+1)$ so that
\[
 \frac{d}{dx}\left[ \frac{g_{\epsilon}(x+1)}{g_{\epsilon}(x)}\right] \leq 0. 
\]
 Then $g(x+1)/g(x)$ is a decreasing function, so that $\lim_{x \to \infty} g(x+1)/g(x)$ exists. 
 
 Since $g(k) \geq a_k$ and the series $\sum a_k$ diverges, then the series $\sum g(k)$ diverges. By the ratio test it follows that 
 \begin{equation}   \label{e:ratiotest}
  \lim_{k \to \infty} \frac{g(k-1)}{g(k)} \leq 1. 
 \end{equation}
 This proves 
 \[
  \lim_{k \to \infty} \frac{\sum_{i=1}^j g(k -i)}{g(k)}  \leq j. 
 \]
 when $j=1$. By induction we assume it holds true for $j$. Then 
 \[
 \begin{aligned}
  \lim_{k \to \infty}\frac{\sum_{i=1}^{j+1} g(k-i)}{g(k)} &= \lim_{k \to \infty} \frac{g(k-{j+1})}{g(k)} +   \frac{\sum_{i=1}^{j} g(k-i)}{g(k)} \\
            &= \lim_{k \to \infty} \prod_{i=1}^{j+1} \frac{g(k-i)}{g(k-(i-1))} + \frac{\sum_{i=1}^{j} g(k-i)}{g(k)}.
  \end{aligned}
 \]
 Then by \eqref{e:ratiotest} as well as the induction hypothesis we conclude that 
 \[
  \lim_{k \to \infty} \frac{\sum_{i=1}^{j+1} g(k-i)}{g(k)} \leq 1 + j.  
 \]
 Finally, we use that $a_{k_l}=g(k_l)$ and $a_{k_{l} - i} \leq g(k_l -i)$ to conclude \eqref{e:sumdiv}. 
\end{proof}

\bibliographystyle{amsplain}
\bibliography{refCorner}

\end{document}